\documentclass{siamltex}
\usepackage{amsmath,amsfonts,amssymb,graphicx,mdwlist,subfigure}

\def\vol{\mathrm{vol\,}}
\def\uno{\mathbf 1}
\def\RR{\mathbb{R}}
\def\t{\text{\itshape{\textsf{T}}}}  

\DeclareSymbolFont{EulerScript}{U}{eus}{m}{n} 
\SetSymbolFont{EulerScript}{bold}{U}{eus}{b}{n} 
\DeclareSymbolFontAlphabet\matheul{EulerScript}

\def\bar{\overline}
\def\emptyset{\varnothing}

\newtheorem{remark}[theorem]{Remark}
\newtheorem{example}[theorem]{Example}

\usepackage[frame,arrow,curve]{xy}
\newenvironment{grafo}[1]{\begin{xy}0;<#1,0cm>: }{\end{xy}}
\newcommand{\Nodo}[5][5mm]{\POS *=<#1>[o]{#3}="#2"+#4(1.5)*{#5}}
\newcommand{\NodoO}[5][5mm]{\POS *=<#1>[o]{#3}*\frm{o}="#2"+#4(1.5)*{#5}}
\newcommand{\arcoD}[3][]{\ar@{-}"#2";"#3"#1}
\newcommand{\loopU}[2][]{\ar@(ur,ul)@{-}"#2";"#2"#1}
\newcommand{\loopD}[2][]{\ar@(dr,dl)@{-}"#2";"#2"#1}

\title{An algebraic analysis of the graph modularity}

\author{Dario Fasino\thanks{Department of Chemistry, Physics,
and Environment, University of Udine, Udine, Italy} 
\and Francesco Tudisco\thanks{Department of Mathematics, University of Rome Tor Vergata, Rome, Italy}}

\begin{document} 
\maketitle

\begin{abstract}
One of the most relevant tasks in network analysis is the detection of community structures, or clustering. 
Most popular techniques for community detection are based on the 
maximization of a quality function called modularity, which in turn is based upon particular quadratic forms associated to a real symmetric modularity matrix $M$, defined in terms of the adjacency matrix and a rank one \textit{null model} matrix. 
That matrix could be posed inside the set of relevant matrices involved in graph theory, alongside adjacency and Laplacian matrices. In this paper we analyze certain spectral properties
of modularity matrices, that are related to the community detection problem.
In particular, we propose a nodal domain theorem for the eigenvectors of $M$; we point out several relations occurring between graph's communities and nonnegative eigenvalues of $M$; and we derive a Cheeger-type inequality for the graph modularity.
\end{abstract}



\thispagestyle{plain}


\section{Introduction}

For the sake of conciseness, we say that 
a complex network is a graph occurring in real life. 
Relevant examples include the Internet and the world wide web, biological and social systems like food webs, economic networks, social networks, communication and distribution networks, and many others \cite{newman-book}. 
Various mathematical disciplines collaborate 
in the analysis and treatment of such complex systems;
and matrix analysis often plays an important role
beside e.g., discrete mathematics and computer science.
Here we consider a clear example of this collaboration,
namely, the subdivision of a network into ``clusters'' (typically connected subnetworks) having certain qualitative properties,
a task which is required in a number of applications. Two main research directions 
can be easily recognized within that topic, 
both having a considerable scientific literature: the \textit{graph partitioning} and the \textit{community detection} (or \textit{clustering}). 

Graph partitioning is the problem of dividing the vertices of a graph into a given number of disjoint subsets of given sizes such that the overall number or weight  of edges between such sets is minimized. The important point here is that the number and sizes of the subsets are, at least roughly, prescribed. For instance, the probably best known example of a graph partitioning problem is the problem of dividing an unweighted graph into two subsets of comparable size, such that the number of edges between them is minimized.

Community detection problems differ from graph partitioning in that the number and size of the subsets into which the network is divided are generally not apriori specified. Instead it is assumed that the graph is intrinsically structured into communities or groups of vertices which are more or less evidently  delimited,
the aim being to reveal the presence and the consistency of such groups. In particular it should be taken into account the possibility that no significant subdivisions exist for a given graph. A comprehensive review of 
methods for the solution of partitioning and 
clustering problems can be found in \cite[Ch.\ 8]{network-analysis} and \cite[Ch.\ 11]{newman-book}; 
see also \cite{santo-fortunato} 
for a good survey.

The question that mainly motivated the present work is indeed related with evaluating the quality of a particular division of a network into communities\footnote{
Our use of the term ``community'' hereafter
makes no reference to its meaning in social sciences
and other disciplines.
We limit ourselves to its common meaning
in the network analysis context. Accordingly, we may also use 
the term ``cluster'' as alternative to ``community'':
\textit{``Clustering is a synonym for the decomposition of a set of entities into natural groups''} \cite{network-analysis}.}  
and providing efficient,
mathematically sound methods and estimates
to locate them. As underlined in \cite{newman-modularity} and \cite{newman-girvan},  
\textit{``a good division of a network into communities is not merely one in which there are few edges between communities; it is one in which there are fewer than expected edges between communities''}. Newman and Girvan  therefore introduced a measure of
the quality of a particular division of a network, which
they call \textit{modularity}. Despite several other quality functions have been proposed in the last ten years
for analogous purposes, 
the modularity is by far the most popular quality function for 
evaluating the quality of a graph partitioning,
and is currently adopted by various successful partitioning algoritms, e.g., the so-called {\em Louvain method}
\cite{Louvain_method}. 
The interesting fact here, and
the issue that has drawn our attention to this topic,
is that the modularity, as well as other 
related graph-oriented topological invariants,
is defined in terms of certain quadratic forms associated to a matrix $M$, called {\em modularity matrix}.
That matrix can be considered as one of the relevant matrices
naturally associated to a graph, together with 
adjacency and Laplacian matrices.

The main aim of this paper is to analyze certain spectral properties of modularity matrices that are relevant to the community detection problem.
In the subsequent part of this Introduction we 
provide the notational and conceptual background for the subsequent discussion. Sections 2 and 3 introduce
the modularity matrix of a graph, its relationships with 
the modularity of a (sub-)graph and with the Laplacian matrix, 
and outline the special role of one of its eigenvalues.
In Section 4 we present a nodal domain theorem for the eigenvectors of modularity matrices. The subsequent sections
are devoted to the analysis of various connections 
between optimal partitionings of a graph and 
nonnegative eigenvalues of its modularity matrix.
Main results are summarized in the concluding Section 8,
which comprises also our final comments and 
possible directions for further research.

\subsection{Notations and preliminary definitions}

To avoid any ambiguity we fix here our notations and some preliminary definitions. We give a brief review of standard concepts from algebraic graph theory that we will use extensively throughout the paper, referring the reader to e.g.,
\cite[Ch.\ 2]{network-analysis} of 
\cite[Ch.\ 6]{newman-book} 
for a careful and succinct introduction to the topic.

From a purely algebraic point of view a graph $G$ consists of a triple $G= (V,E,\omega)$ where $V$ is the set of vertices (or nodes), $E$ is the set of edges and formally is a subset of $V\times V$, and $\omega:E\to \RR^+$ is a nonnegative weight function defined over $E$, 
representing the strength of the relation modeled by the edges. 
We shall always assume that a graph $G$ 
is finite, simple, connected, not oriented. 
We always identify $V$ with $\{1,\dots,n\}$. We use the simpler notation $G = (V,E)$ when $\omega(ij) = 1$, that is, edges are not weighted.



If not otherwise specified, the symbol $A$ will always denote 
the adjacency matrix of $G$, that is, $A\equiv(a_{ij})$
such that $a_{ij} = \omega(ij)$ iff $ij \in E$, and $a_{ij}=0$ otherwise. In particular, $A$ is a symmetric, irreducible, componentwise nonnegative matrix.
For the sake of clarity, further definitions are listed
hereafter:

\begin{itemize}
\item If $ij  \in E$ we write $i\sim j$ and say that $i$ and $j$ are adjacent.

\item For any $i\in V$, $d_i$ denotes its (generalized)
degree, $d_i=\sum_{j: ij\in E}\omega(ij)$. 
Moreover, we let 
$d = (d_1,\ldots,d_n)^\t$
and $D = \mathrm{Diag}(d_1,\ldots,d_n)$.

\item For any $S\subseteq V$ we denote by
$\bar S$ the complement $V \setminus S$, and let $\vol S = \sum_{i \in S}d_i$ be the {\em volume} of $S$. Correspondingly,  
$\vol G= \sum_{i \in V}d_i$ denotes the volume of the whole graph.

\item 
A partition of $V$ is a collection of subsets 
$\matheul{P} = \{S_1,\ldots,S_k\}$ such that 
$\cup_iS_i = V$ and $S_i\cap S_j = \varnothing$ for $i\neq j$.

\item For $S \subseteq V$, we denote by $A(S)$ the principal submatrix of $A$ made by the rows and columns 
whose indices belong to $S$. 
Moreover,
we denote by $G(S)$ the subgraph induced by the vertices in $S$, that is the subgraph of $G$ whose adjacency matrix is $A(S)$.

\item $\uno$ denotes the vector of all ones whose dimension depends on the context.

\item The cardinality of a set $S$ is denoted by $|S|$. 
In particular, 
$|V|=n$.

\item For any $S \subseteq \{1,\dots,n\}$ we let $\uno_S$ be its \textit{characteristic vector}, defined as $(\uno_S)_i =1$ if $i \in S$ and $(\uno_S)_i=0$ otherwise.

\item
For any subsets $S,T\subseteq V$ let
$E(S,T)$ be the set of edges joining vertices in $S$
with vertices in $T$; and let
$$
   e(S,T) = \uno_S^\t A \uno_T =
   \sum_{i \in S}\sum_{j \in T}\omega(ij) .
$$ 
Note that, if $G = (V,E)$ is unweighted and loopless then
$e(S,T) = 2|E(S,T)|$.
For simplicity, we 
use the shorthands $e_{\mathrm{in}}(S) = e(S,S)$
and $e_{\mathrm{out}}(S) = e(S,\bar S)$,
so that we have also
$$
   \vol S = e_{\mathrm{in}}(S) + e_{\mathrm{out}}(S) .
$$

\item For a matrix $A$ and a vector $x$, we write $A\geq O$ or $x\geq 0$ (resp. $A>O$ or $x>0$) to denote  componentwise nonnegativity (resp., positivity);
and $\rho(A)$ denotes the spectral radius of $A$. 

\item If $X$ is a symmetric matrix then its eigenvalues are ordered as $\lambda_1(X)\geq \cdots\geq \lambda_n(X)$, unless otherwise specified.

\end{itemize}

We will freely use familiar properties of
eigenvalues of symmetric matrices,
and fundamental results in
Perron-Frobenius theory, see e.g., 
\cite{positive-book,bhatia}. 
For completeness, we recall hereafter 
some important facts concerning the symmetric eigenvalue problem:
\begin{itemize}
\item 
Let $A\in\mathbb{R}^{n\times n}$ be a symmetric matrix 
and let $Z\in\mathbb{R}^{n\times (n-k)}$ be a matrix 
with orthonormal columns.
Then, for all $i = 1,\ldots,n-k$,
\begin{equation}   \label{eq:interlacing}
   \lambda_i(A) \geq \lambda_i(Z^\t AZ) \geq \lambda_{i+k}(A) .
\end{equation}
\item 
Let $A\in\mathbb{R}^{n\times n}$ be a symmetric matrix 
and let $B\in\mathbb{R}^{(n-k)\times (n-k)}$ 
be a principal submatrix of $A$. Then, for all $i = 1,\ldots,n-k$,
\begin{equation}   \label{eq:interlacing2}
   \lambda_i(A) \geq \lambda_i(B) \geq \lambda_{i+k}(A) .
\end{equation}
\item
Let $A$ be a real symmetric matrix of order $n$
and $v\in\mathbb{R}^n$. Then, for $i = 1,\ldots,n-1$,
\begin{equation}   \label{eq:Weyl}
   \lambda_i(A) \geq \lambda_{i+1}(A+vv^\t)
   \geq \lambda_{i+1}(A) .
\end{equation}
\end{itemize}

\subsubsection{The modularity matrix}


The \textit{modularity matrix} of the graph is defined as follows:
\begin{equation}   \label{eq:M}
   M = A- \frac{1}{\vol G}  dd^\t.
\end{equation}
Modularity matrices have been introduced originally 
for unweighted graphs $G = (V,E)$; in that case,
the number $a_{ij}$ indicates the presence of
an edge between nodes $i$ and $j$, 
whereas $d_id_j/\vol G$ 
estimates the expected number of edges
between vertices $i$ and $j$, if edges in the graph were placed with an uniformly random distribution,
according to the given degree sequence $d_1,\ldots,d_n$. 
Therefore the $(i,j)$-entry 
$m_{ij} = a_{ij} - (d_id_j)/(\vol G)$ of $M$ measures the 
disagreement between the expected number 
and the actual number of edges joining $i$ and $j$. 
It is a common practice to extend rather informally
this definition to any weighted
graph $G = (V,E,\omega)$. In the next paragraph, 
we outline a formal  
justification of this rather natural extension.

\begin{remark}
In any unweighted graph, 
the number $d_id_j/\vol G$ is always an upper bound
on the probability that $ij \in E$, assuming that 
edges are placed in $G$ independently 
at random, conditionately to the given degrees. 
In fact, that number is the first term in a sign-alternating series expressing the actual probability, where
successive terms represent the probability that 
$i$ and $j$ are connected by multiple edges.
If $d_id_j/\vol G \ll 1$ then the alternating series is
rapidly convergent, and the bound is a good approximation to the true value. On the other hand, 
graphs of practical interest may contain node pairs with
$d_id_j/\vol G > 1$ and weighted edges.
In any case, a more principled motivation of the 
rank-one correction in the modularity matrix, 
which carries over the weighted graph case,
relies on the so-called Chung-Lu random graph model,
to be recalled in the next paragraph.
\end{remark}

\subsubsection{The Chung-Lu random graph model}

The Chung-Lu random graph model is one of the most
widespread and successful models for the analysis
of large graphs with general degree distributions.
Let $w = (w_1,\ldots,w_n)^\t > 0$ be a vector  
fulfilling the condition $\max_iw_i^2 < \sum_{i=1}^nw_i$. We say that a graph $G = (V,E)$ follows the 
Chung-Lu random graph model with parameter $w$,
denoted by $\mathcal{G}(w)$, if the existence of the edge
$ij\in E$ is determined by an independent Bernoulli trial
with probability
$p_{ij} = w_iw_j/(\sum_{i=1}^nw_i)$.
That model has been popularized in
\cite[Ch.\ 5]{chung-lu-networks}; and various 
statistical properties have been described e.g., in 
\cite{arcolano:moments,higham-random-model}.
A basic and very useful property of this model
is that, if $G = (V,E)$
is a random graph drawn from $\mathcal{G}(w)$,
then the expected degree of $i\in V$ is exactly $w_i$.
Consequently, if only the degree vector $d$
is known, it is reasonable to assume $w = d$. 
Actually, this equality leads to an 
asymptotically unbiased estimator of $w$ \cite{arcolano:moments}.
Hereafter, we propose a generalization of the Chung-Lu model
which is convenient for working with weighted graphs. 

\begin{definition}   \label{def:wCL}
Let $w = (w_1,\ldots,w_n)^\t > 0$, and
let $X(p)$ be a nonnegative random variable parametrized
by the scalar parameter $p\in[0,1]$, whose expectation is 
$\mathbb{E}(X(p)) = p$. 
We say that a weighted graph $G = (V,E,\omega)$ follows the 
{\em $X$-weighted Chung-Lu random graph model} $\mathcal{G}(w,X)$ 
if, for all $i,j\in V$, 
$\omega(ij)$ are independent random variables
distributed as $X(p_{ij})$ where
$p_{ij} = w_iw_j/\sum_{i=1}^nw_i$,
with the convention that 
$ij\in E \Leftrightarrow \omega(ij) > 0$, that is,
edges with zero weight are removed from $G$.
\end{definition}

We point out that $\mathcal{G}(w)$ is the special case
of $\mathcal{G}(w,X)$ where $X(p)$ is the Bernoulli trial 
with success probability $p$. On the other hand, if $X(p)$
has a continuous part, then $\mathcal{G}(w,X)$
may contain graphs with generic weighted edges.
In any case, as in the original Chung-Lu model, 
if $G$ is a random graph drawn from $\mathcal{G}(w,X)$ then 
the expected degree of node $i$ is
$$
   \mathbb{E}(d_i) = \sum_{j=1}^n \mathbb{E}(\omega(ij))
   = \sum_{j=1}^n p_{ij}
   = w_i .
$$


\subsubsection{The Laplacian matrix}

The modularity matrix \eqref{eq:M}
is a rank one perturbation of the adjacency matrix, which is still symmetric but looses the nonnegativity of its entries. The kernel of $M$ is nontrivial and, indeed, $\uno$ always is a nonzero element in $\ker M$. This is reminiscent of another key matrix associated to a graph $G$: the Laplacian matrix. Such matrix is defined as $L = D-A$, where $D$ denotes the diagonal matrix with diagonal entries $d_1,\dots,d_n$.
A huge literature has been developed around $L$, its 
spectral properties, and their connections with 
combinatorial and topological properties of $G$,
see e.g., \cite{chung,expanders} and references therein; 
in fact, this matrix can be thought as a discrete version 
of the Laplacian differential operator,  
under many respects. 

The bilinear form associated to $L$ admits the expression
\begin{equation}   \label{eq:vLv}
   v^T L v = \sum_{ij \in E}
   \omega(ij) (v_i-v_j)^2, 
\end{equation}
where the sum ranges over all edges in the graph, each edge being counted only once. Thus,
$L$ is symmetric and positive semidefinite; zero always is a eigenvalue of $L$, with associated eigenvector $\uno$, 
and that eigenvalue
is simple if and only if $G$ is connected.  
Conventionally the eigenvalues of $L$ are ordered from  
smallest to largest; for a connected graph, 
$0=\lambda_1(L)< \lambda_2(L)\leq \dots \leq \lambda_n(L)$.

\subsubsection{Nodal domains}

The study of the spectral properties of the Laplacian matrix has 
originated
one of the best known methods for graph partitioning, the \textit{spectral partitioning} \cite[\S 11.5]{newman-book}. The idea was pioneered by Fiedler 
in \cite{fiedler-connectivity,fiedler-vector}, where he observed that a strong relation exists among connectivity 
properties of $G$, the second smallest eigenvalue of $L$ (the smallest one being zero), and the changes of signs of the entries of any eigenvector relative to such eigenvalue. 
Following Fiedler's works, the number $\lambda_2(L)$
is usually called {\em algebraic connectivity} of $G$
and denoted by $a(G)$; furthermore, it is a well established practice to call {\em Fiedler vector} any eigenvector associated to it.

Let us recall a couple of definitions and
relevant results. 
Inspired by Courant's nodal domains theorem (which bounds the number of nodal domains of eigenfunctions of
the Laplacian operator on smooth Riemannian manifolds), 
nodal domains induced by a real vector $u$ are commonly defined as follows:\footnote{
Unlike their continuous analogous, in 
the present 
context nodal domains are located by sign variations rather than zero values.  
Therefore some authors call them \textit{sign domains}
\cite{nodal-domain-theorem}. We prefer to maintain the ``classical" terminology.}

\begin{definition}   \label{def:strong_nd}
Let $0 \neq u \in \RR^n$. 
A subset $S \subseteq V$ is a \emph{strong nodal domain} of $G$ induced by $u$ if the subgraph $G(S)$ induced on $G$ by $S$ 
is a (maximal) connected component of either
$\{i: u_i > 0\}$ or $\{i: u_i < 0\}$.
\end{definition}

\begin{definition}   \label{def:weak_nd}
Let $0 \neq u \in \RR^n$. 
A subset $S \subseteq V$ is a \emph{weak nodal domain} of $G$ induced by $u$ if the subgraph $G(S)$ induced on $G$ by $S$ 
is a (maximal) connected component of either
$\{i: u_i \geq 0\}$ or $\{i: u_i \leq 0\}$ and contains at least one node $i$
where $u_i \neq 0$.
\end{definition}

Actually, the previous definitions are a slight modification of the terminology used in e.g., 
\cite{nodal-domain-theorem,duval-nodal-domains},
but their meaning is unchanged.
For any connected graph $G$, $\lambda_1(L) = 0$ is simple and 
has $\uno$ as associated eigenvector.
It clearly follows that the only possible nodal domain for $\lambda_1(L)$ is $G$ itself. On the other hand, since $L$ is real and symmetric, each other eigenvector of $L$ can be chosen to be real and orthogonal to $\uno$, that is, any eigenvector $u$ of $L$ 
that is not constant 
has at least two components of different signs. Therefore any such $u$  has at least two nodal domains. 
Fiedler noted in \cite[Cor.\ 3.6]{fiedler-vector}
that the weak nodal domains induced by 
any eigenvector associated to $a(G)$
are at most two, and thus are exactly two. 
Many authors derived analogous results for the other eigenvalues of $L$ afterward
\cite{nodal-domain-theorem, duval-nodal-domains, powers-graph-eigenvector}. 
The following {\em nodal domain theorem} summarizes their work:

\begin{theorem}\label{teo:laplacian-strong-domains}
Let $L$ be the Laplacian matrix of a connected graph.
Let $\lambda$ 
be an eigenvalue of $L$ and let $u$
be an associated eigenvector.
Let $\ell$ and $\ell'$ be the number of eigenvalues of $L$
that are not larger than $\lambda$ and
strictly smaller than $\lambda$, respectively, counted with their multiplicity.
Then $u$ induces at most $\ell$ strong nodal domains 
and at most $\ell'+1$ weak nodal domains.
\end{theorem}

\section{Modularity of a subgraph}

A central problem in graph clustering is to look for a quantitative definition of community.
Although all authors agree that a community should be a 
connected group of nodes that is more densely connected among each other than with the rest of the network,
as a matter of fact no definition is universally accepted.
A variety of merit functions 
to quantify the strength of a subset $S\subset V$
as a community in $G$ is listed in \cite[Ch.\ 8]{network-analysis}; 
all of them are essentially based on a trade-off
between the total weight of edges insisting on vertices in $V$
(which should be ``large'') 
and the one of the edges connecting vertices in $V$ 
with vertices outside $V$  (which should be ``small'',
for a ``good'' community).

Fortunato in its comprehensive report \cite{santo-fortunato} 
classifies various definitions of community
according to whether they are based on graph-level properties, subgraph-level properties, or vertex similarity,
and underlines that the global definition  
based on the modularity quality function introduced by Newman and Girvan in \cite{newman-girvan} is by far the most popular definition. Their definition can be informally stated as follows: 
A subset of vertices $S \subseteq V$ forms a community
if the subgraph $G(S)$ 
contains a larger number of edges than expected.
Obviously, such statement is not rigorous, until one defines
the probability distribution underlying the concept of
``expected number''. Doubtless, the most simple and natural guess
is to assign an equal probability 
to the connection between any two nodes in the network. 
The corresponding random graph model 
is known as Erd\"os-R\'enyi model. That model 
is at the basis of 
various successful approaches to community detection
\cite{multires-2,multires-1,TVDN_CPM}.
In this work, we follow 
\cite{newman-eigenvectors,newman-modularity,newman-girvan}
and assume, instead, the Chung-Lu random graph model
with parameter $d$ 
as reference.

Given a graph $G = (V,E,\omega)$, consider a subset of 
vertices $S\subseteq V$. 
For graphs following the (weighted) 
Chung-Lu model with parameter $d$, the overall weight 
of edges joining vertices in $S$ 
can be estimated by
$$
   \sum_{i\in S} \sum_{j\in S} \frac{d_id_j}{\vol G}
   = \frac{(\vol S)^2}{\vol G} .
$$
Consequently, we define the {\em modularity} 
of $S$ as 
\begin{equation}   \label{modularity}
   Q(S) = e_{\mathrm{in}}(S) - (\vol S)^2/\vol G . 
\end{equation}
If that difference
is positive then there is a clear indication that the subgraph
$G(S)$ contains ``more edges'' 
than expected from 
the reference model. This fact can be considered as 
a clue (apart from connectedness)
that $S$ is a closely knit set of vertices and
as such, a possible community inside $G$.

An easy computation exploiting the identities 
$\vol S = e_{\mathrm{in}}(S) + e_{\mathrm{out}}(S)$  
and $\vol G - \vol S = \vol \bar S$ reveals that
\begin{align}   \label{modularity-3}
   Q(S) \notag
   &=\textstyle{ \vol S - e_{\mathrm{out}}(S) - 
   \frac{(\vol S)^2}{\vol G}}    \\ 
   &= \notag \textstyle{\vol S (1 - 
   \frac{\vol S}{\vol G}) - e_{\mathrm{out}}(S)}   \\ 
   &= \textstyle{\frac{\vol S \, \cdot \, \vol \bar S}{\vol G}
   - e_{\mathrm{out}}(S) } .
\end{align}
Such relation shows that $Q(S)=Q(\bar S)$. Therefore, modularity is a quality of the \textit{cut} $\{S,\bar S\}$ rather than of $S$ itself. 
Moreover it reveals that $Q(S)$ is large when both $S$ and its complement $\bar S$ have comparable volumes 
(in fact $\vol S \,\vol \bar S / \vol G$ is large when 
$\vol S \approx \vol \bar S \approx \frac12 \vol G$) and the 
overall weight 
of edges elapsing between $S$ and $\bar S$ is small. 
Consequently, \eqref{modularity-3} 
bares that the modularity $Q(S)$ 
shares the structure of virtually all reasonable
clustering indices 
\cite[Ch.\ 8]{network-analysis}, consisting of the difference
between $\vol S\, \vol \bar S / \vol G$, which 
is a term measuring the density of the ``clusters'' $S$ and $\bar S$, and
$e_{\mathrm{out}}(S)$, 
which quantifies the sparsity of their connection.
Furthermore, the resulting equalities $Q(\emptyset) = Q(V)=0$ 
formalize the common understanding that neither the emptyset nor the whole graph constitute a community. 

It is almost immediate to recognize that
$e_{\mathrm{in}}(S) = \uno_S^\t A \uno_S$ and 
$\vol S = \uno_S^\t d$. 
Hence, we can express
the modularity \eqref{modularity} in terms of the modularity matrix \eqref{eq:M} as follows:
\begin{equation}\label{modularity-via-M}
   Q(S) = \uno_S^\t M\uno_S .
\end{equation}

\begin{remark}
In principle other vectors can be chosen in place of $d$ 
inside \eqref{eq:M},
depending on the 
\textit{null model} one is assuming 
for the distribution of the edges in $G$. 
For example, if 
$G$ is unweighted and 
the null model assumed is the Erd\"{o}s-R\'enyi random graph model, in which 
every edge has probability $p$ to appear,
then 
the appropriate definition for the modularity matrix of $G$ would be $M = A - p\uno \uno^\t$ with 
$p = \vol G/n^2$, so that $Q(V) = \uno^\t M\uno = 0$. 
In this case, the resulting modularity matrix 
allows us to express 
by means of a formula analogous to \eqref{modularity-via-M}
certain modularity-type merit functions based on $2$-state Potts Hamiltonian functions 
adopted in, e.g., \cite{multires-1,TVDN_CPM}.
\end{remark}

In a somehow heuristic way at this stage,  we see from \eqref{modularity-via-M} that the existence of a subset $S\subseteq V$ having positive modularity is related with the positive eigenvalues of $M$ and their corresponding eigenspaces. 
In fact, if $\mathbb{F}_n =\{0,1\}^n$ 
is the set of binary $n$-tuples, the search of a maximal modularity subgraph is formalized by the optimization problem
\begin{equation}\label{bin}
  \max_{x \in \mathbb{F}_n} x^\t Mx .
\end{equation}
The problem as is stated is clearly NP-complete,
so a standard and widely used procedure is to move to a continuous relaxation, for example,
\begin{equation}   \label{cont}
   \max_{\substack{x \in \RR^n \\ x^\t x=1}} x^\t Mx ,
\end{equation}
which is solved by an eigenvector associated to 
the largest eigenvalue of $M$, properly normalized. 
Once a solution $\tilde x$ for the latter problem \eqref{cont} is computed, the sign vector 
$s = \text{sign}(\tilde x)$ is chosen as an approximate solution for \eqref{bin}. 
Note that such $s$ realizes the best approximation 
to $\tilde x$ in the $L^p$ sense, that is 
$\|s-\tilde x\|_p = \min_{x \in \mathbb{F}_n}\|x-\tilde x\|_p$, 
for $p\in [1,\infty]$. 
The spectral analysis of $M$ and of the maximal subgraphs 
induced by the change of signs in its eigenvectors (nodal domains) is the central topic of Sections \ref{sec:nodal}.

In what follows, we adopt from 
\cite{newman-modularity,newman-girvan} the following definitions:

\begin{definition}   \label{def:community} 
A {\em module} in a given graph $G$ is a subgraph
having positive modularity. A graph is {\em indivisible}
if it has no modules, and {\em divisible} otherwise.
\end{definition}

Probably, the main reason of the success of modularity
as a quantitative measure of community strength is the fact that
modules having significant size and 
modularity are typically decent indicators of community structure.

\begin{remark}   \label{rem:indivisible}
Cliques and star graphs are 
indivisible graphs.
On the other hand, indivisible graphs are rather scarce.
Indeed, a simple computation based on the formula 
\eqref{modularity} shows that, if $i,j\in V$ are two 
vertices joined by an edge, and 
$$ 
   d_i + d_j < \sqrt{2\,\omega(ij) \vol G} ,
$$
then $Q(\{i,j\}) > 0$.   
Consequently, a graph is divisible
if it has at least one edge fulfilling the previous inequality,
a condition which is easily met in practice.
\end{remark}

\section{The algebraic modularity of a graph}

Since the pioneering works by Fiedler
\cite{fiedler-connectivity,fiedler-vector}
the algebraic connectivity of a connected graph 
$G$ is classically defined
as the smallest positive eigenvalue of its Laplacian matrix:
$$
   a(G) = \min_{x^\t\uno = 0} \frac{x^\t L x}{x^\t x} .
$$
Analogously, we can define the {\em algebraic modularity of $G$}
as
\begin{equation}   \label{eq:def_m(G)}
   m(G) = \max_{x^\t\uno = 0} \frac{x^\t M x}{x^\t x} .
\end{equation}
Differently to \eqref{cont}, 
any vector $x$ attaining the maximum in \eqref{eq:def_m(G)}
must have entries with opposite signs. 
We will see afterward that $m(G)$ plays a relevant role
in the community detection problem, exactly in the same way
as $a(G)$ with respect to the partitioning problem.
Furthermore, in tandem with Definition \ref{def:community},
it is rather natural to say that $G$ is
{\em algebraically indivisible} if
its modularity matrix has no positive eigenvalues.
For example, 
cliques and star graphs are 
algebraically indivisible graphs.

\begin{remark}   \label{rem:m(G)}
The number $m(G)$ is the largest eigenvalue of $M$
after deflation of the subspace $\langle\uno\rangle$,
which is an invariant subspace associated to the eigenvalue $0$.
More precisely, we have $\lambda_1(M) = \max\{m(G),0\}$.
Hence, we can say that $m(G) = \lambda_1(M)$ 
if and only if $m(G) \geq 0$.
\end{remark}

We point out that any algebraically indivisible graph is  
indivisible as well.
Indeed, the existence of a subgraph $S$ having positive modularity implies that $M$ has at least one positive eigenvalue: $\lambda_1(M) \geq \uno^\t_S M\uno_S/\uno^\t_S\uno_S = Q(S)/|S| > 0$.
We shall explore in greater detail in Section \ref{sec:6}
the relationship between 
divisibility of $G$ 
and positive eigenvalues of $M$. For the moment,
the following argument shows that a better bound than $Q(S)\leq |S|\lambda_1(M)$ can be derived:

\begin{lemma}   \label{lem:g_G'}
For any $S\subseteq V$ we have $Q(S) \leq m(G)|S||\bar S|/n$.
\end{lemma}

\begin{proof}
Let $\alpha = |S|/n$. Then, the vector $\uno_S-\alpha\uno$
is orthogonal to $\uno$ and moreover,
$$
   (\uno_S-\alpha\uno)^\t(\uno_S-\alpha\uno) = 
   (\uno_S-\alpha\uno)^\t \uno_S = 
   |S| - \alpha |S| = \frac{|S||\bar S|}{n} .
$$
Recalling that $M\uno = 0$ and the definition \eqref{eq:def_m(G)}
we have
$$
   Q(S) = \uno_S^\t M\uno_S = 
   (\uno_S-\alpha\uno)^\t M(\uno_S-\alpha\uno) \leq
   m(G) (\uno_S-\alpha\uno)^\t(\uno_S-\alpha\uno) ,
$$
and we complete the proof.
\end{proof}

It is worth noting that the modularity matrix
$M$ can be expressed as the difference of two Laplacian matrices.
Indeed, 
\begin{equation}   \label{eq:M_L}
   M = A - D + D -  dd^\t/(\uno^\t d) = 
   L_0 - L , 
\end{equation}
where $L = D-A$ is the 
Laplacian matrix of $G$ and $L_0 = D-  dd^\t /(\uno^\t d)$
can be regarded as the Laplacian matrix of the 
complete graph $G_0 = (V,V\times V,\omega_0)$
where the weight $\omega_0(ij) = d_id_j/\uno^\t d$
is placed on the edge $ij$. 
Thus, in some sense, $G_0$ represents the ``average graph''
in the Chung-Lu model with parameter $d$.

The formula \eqref{eq:M_L}
yields a decomposition of $M$ in terms of two positive semidefinite matrices. A noticeable consequence 
of the Courant-Fischer theorem is the following 
set of inequalities, 
relating algebraic connectivity and modularity of $G$,
and whose simple proof is omitted for brevity:
$$ 
   d_{\min} - a(G) \le a(G_0) - a(G)
   \le m(G) \le d_{\max} - a(G) ,
$$
where $d_{\min}$ and $d_{\max}$ denote the smallest and largest degree of vertices in $G$, respectively. 
Consequently, a necessary condition for $G$ being 
algebraically indivisible is
$a(G_0) \le a(G)$. 
By a result by Fiedler 
\cite{fiedler-connectivity}, whose proof extends immediately
to weighted graphs,  
$a(G) \le [n/(n-1)]d_{\min}$.
Hence, $m(G) \ge -d_{\min}/(n-1)$. 
This lower bound is attained by a clique,
thus it is sharp.

\section{Modularity nodal domains}\label{sec:nodal}

As recalled in Definition \ref{def:strong_nd}
and Definition \ref{def:weak_nd},
any vector $u \in \RR^n$ induces some nodal domains over 
$G$, that is some maximal connected subsets of the vertices $V$ related with sign changes inside $u$. 
Hereafter, we consider nodal domains induced by eigenvectors of the modularity matrix of the graph, which we call \textit{modularity nodal domains}. 
The aim of this section is to derive a nodal domain theorem 
analogous to Theorem \ref{teo:laplacian-strong-domains}
for the modularity nodal domains, contributing to the analysis and the improvement of the spectral-based methods for community detection, proposed by Newman and Girvan \cite{newman-girvan} and well summarized in 
\cite{santo-fortunato} and \cite[Ch.\ 11]{newman-book}. 

We will say that a nodal domain $S\subset V$ 
induced by a vector $u$
is {\em positive} or {\em negative}, according to
the sign of $u$ over $S$.
If $S_1$ and $S_2$ are two nodal domains, we say that $S_1$ is adjacent to $S_2$, in symbols $S_1 \approx S_2$, if there exists $i \in S_1$ and  $j \in S_2$ such that $i \sim j$. The maximality of the nodal domains therefore implies that a necessary condition for $S_1 \approx S_2$ is that  $S_1$ and $S_2$ have different signs. 

Given a real vector $u \neq 0$ the following properties on the  nodal domains it induces are not difficult to be observed;
some of them are borrowed from \cite{nodal-domain-theorem}:
\begin{enumerate}

\item[\textit{P1.}] 
In any nodal domain there exists at least one node where $u$ is nonzero.
Moreover, if $S_1$ and $S_2$ are weak nodal domains such that $S_1 \cap S_2 \neq \emptyset$ then $S_1$ and $S_2$ have opposite sign and $u_i=0$ for any $i \in S_1 \cap S_2$.

\item[\textit{P2.}] Let $A$ be the adjacency matrix of $G$. If $S \subseteq V$ is a (strong or weak)
nodal domain, then  $G(S)$ is connected and the  principal submatrix $A(S)$  is irreducible. Therefore, since two nodal domains of the same sign can not be adjacent, for any vector $u$
there exists a labeling of the vertices of $V$ such that the adjacency matrix $A$ of $G$ has the form 
\begin{equation}\label{PAP}
   A = \begin{pmatrix}
       A_+ & B & C  \\
       B^\t & A_- & D \\
       C^\t & D^\t & A_0 \end{pmatrix}
\end{equation}
where rows and columns of $A_+$, $A_-$, and $A_0$ 
correspond to entries in $u$ that are positive, negative, and zero,
respectively,
and $A_+$ and $A_-$ are the direct sum of overall $s$ irreducible matrices, $s$ being the number of strong nodal domains.

\item[\textit{P3.}] If $S_1$ and $S_2$ are adjacent weak nodal domains, then there exists $i \in S_1$ and $j \in S_2\setminus S_1$ such that $i \sim j$ and $u_j \neq 0$. In fact, if $S_1 \cap S_2 = \emptyset$ then the assertion follows by definition. 
(If $i\sim j$ and $u_j = 0$ then $j\in S_1\cap S_2$.)
Whereas if $S_1 \cap S_2 \neq \emptyset$ then, by 
property \textit{P1},
there must be at least a pair of vertices $i,j \in S_1 \cup S_2$ for which 
$i \in S_1\cap S_2$ (whence $u_i = 0$), $i \sim j$, and $j \in S_2\setminus S_1$
(so that $u_j \ne 0$); otherwise, there would be 
no edge joining $S_1\cap S_2$ and $ S_2\setminus S_1$,
contradicting the hypothesis that $G(S_2)$ is connected. 
\end{enumerate}

The following theorem, which is a slight generalization of 
\cite[Thm.\ 2.1]{fiedler-vector} and 
\cite[Thm.\ 1]{powers-graph-eigenvector}, 
is the key for deriving a nodal domain theorem for the modularity eigenvectors. 
We stress that such theorem and its corollaries hold for any undirected simple graph, that is, loops and weighted edges are possibly allowed.

\begin{theorem}\label{teo:main}
Let $A$ be the adjacency matrix of a simple, connected graph $G$. 
Let $\lambda \in \RR$ and $u \in \RR^n$ be such that 
at least two entries of $u$ have opposite signs
and 
$Au\geq \lambda u$, in the componentwise sense. Let $\ell$ and $\ell'$ be respectively  the number of eigenvalues of $A$ that are greater than or equal to  $\lambda$ and the number of eigenvalues that are strictly greater than $\lambda$, counted with their multiplicity. Then $u$ induces at most $\ell$ positive strong nodal domains 
and at most $\ell'$ positive weak nodal domains.
\end{theorem}

\begin{proof}
Let $s\geq 1$ be the number of positive strong 
nodal domains induced by $u$. 
Due to property \textit{P2} above, 
we can assume without loss in generality that
the vector $u$ can be partitioned into $s+1$ subvectors,
$u=(u_1,\dots,u_s,u_{s+1})^\t$ such that 
$u_i>0$, for $i=1,\dots,s$,
$u_{s+1}\leq 0$ 
and $A$ is conformally partitioned as
$$
  A = \begin{pmatrix}
  A_1 & & & B_1 \\ 
  & \ddots & & \vdots \\
  & & A_{s} & B_s \\
  B_1^\t & \cdots & B_s^\t & B_{s+1}
  \end{pmatrix},
$$
where $A_i$ are nonnegative and irreducible,
since they are the adjacency matrices of connected graphs. 
By hypothesis, $A_i u_i + B_i u_{s+1} \geq \lambda u_i$
 for $i=1,\dots,s$. Therefore $A_i u_i \geq \lambda u_i - B_i u_{s+1}\geq \lambda u_i$ and, by Perron-Frobenius theorem
 we have 
$$
  \rho(A_i)
   = \max_{x\ne 0} \frac{x^\t A_i x}{x^\t x}
   \geq \frac{u_i^\t A_i u_i}{u_i^\t u_i}\geq \lambda . 
$$
This implies that $A_i$ has at least one eigenvalue
not smaller than $\lambda$, for $i=1,\dots,s$. By 
eigenvalue interlacing inequalities 
\eqref{eq:interlacing2} 
we conclude that $A$ has at least $s$ eigenvalues greater than or equal to $\lambda$, whence $s\leq \ell$. This proves the first inequality in the claim. 
 
The second one can be proved analogously. As for the strong domains, two positive weak nodal domains can not overlap, therefore there exists a labeling of $V$  such that 
$A$ admits the block form
 $$
  A = \begin{pmatrix}
  A_1 & & & B_1 \\ 
  & \ddots & & \vdots \\
  & & A_{w} & B_w \\
  B_1^\t & \cdots & B_w^\t & B_{w+1}
  \end{pmatrix},
$$
where $w$ is the number of weak positive nodal domains,
and the vector $u$ is partitioned conformally as
$u=(u_1,\dots,u_w,u_{w+1})^\t$ where 
$u_i\geq 0$ for $i=1,\dots,w$ and $u_{w+1}\leq 0$.
In fact, the entries in $u_{w+1}$ correspond to nodes belonging to the 
complement of the union of all positive weak nodal domains,
and $u$ may vanish also on some of those nodes.
Nevertheless,
property \textit{P3} above 
imply that each  $B_i$ contains at least one nonzero entry, and $B_iu_-\leq 0$ 
with strict inequality in at least one entry. 
For any fixed $i = 1,\ldots,w$ let $x_i$ be 
a Perron eigenvector of $A_i$, 
$A_ix_i = \rho(A_i)x_i$, with positive entries.
Hence $x_i^\t u_i > 0$ and $x_i^\t B_iu_{w+1} < 0$.
From the inequality 
$A_i u_i \geq \lambda u_i - B_i u_{w+1}$ we obtain
$$
   \rho(A_i) x_i^\t u_i =
   x_i^\t A_i u_i \geq
   \lambda x_i^\t u_i - x_i^\t B_i u_{w+1}
   > \lambda x_i^\t u_i
$$
for $i=1,\dots, w$. Again by the 
eigenvalue interlacing \eqref{eq:interlacing2} 
we see that $A$ has at least $w$ eigenvalues strictly greater that $\lambda$, concluding that $w\leq \ell'$.
\end{proof}

Note that in the preceding theorem $\lambda$ may not be 
an eigenvalue of $A$, in which case $\ell = \ell'$.
If $\lambda$ is an eigenvalue of $A$ then the difference $\ell - \ell'$ equals its algebraic/geometric multiplicity.

\subsection{A modularity nodal domain theorem}

A direct consequence of Theorem \ref{teo:main} 
is the following result concerning the nodal domains
of eigenvectors of modularity matrices, as announced:

\begin{theorem} 
\label{thm:modularity-nodal-theorem}
Let $\lambda$ be an eigenvalue of $M$ and let 
$u$ be an associated eigenvector, 
oriented so that $d^\t u\geq 0$. 
Let $\ell$ and $\ell'$ be respectively  the number of eigenvalues of $M$ which are greater than or equal to  $\lambda$ and the number of eigenvalues which are strictly greater than $\lambda$, counted with their multiplicity. 
If at least two entries of $u$ have opposite signs
then $u$ induces at most $\ell+1$ \emph{positive} strong nodal domains 
and at most $\ell'+1$ \emph{positive} weak nodal domains.
\end{theorem}

\begin{proof}
By hypotheses, $\ell \geq \ell'+1$ and the eigenvalues of $M$ fulfill
$$
   \lambda_{\ell'}(M) > 
   \underbrace{\lambda_{\ell'+1}(M) = \ldots 
   = \lambda_{\ell}(M)}_{\hbox{\scriptsize eigenvalues equal to $\lambda$}}
   > \lambda_{\ell+1}(M) ,
$$
the first (last) inequality being missing if $\ell' = 0$
($\ell = n$, respectively).
Since $A-M$ is a positive semidefinite rank-one matrix,
inequalities \eqref{eq:Weyl}
imply the following interlacing
between the eigenvalues of $A$ and of $M$:
$$
   \lambda_1(A) \geq \lambda_1(M) \geq
   \lambda_2(A) \geq \lambda_2(M) \geq \ldots \geq \lambda_n(M) .
$$
By inspecting the preceding inequalities
we get that 
\begin{itemize}
\item 
$\lambda_{\ell}(A) \geq \lambda > \lambda_{\ell+2}(A)$; 
thus $\ell \leq |\{i : \lambda_i(A) \geq \lambda\}| \leq \ell +1$.
\item 
$\lambda_{\ell'}(A) > \lambda \geq \lambda_{\ell'+2}(A)$; thus $\ell' \leq |\{i : \lambda_i(A) > \lambda\}| \leq \ell' +1$.
\end{itemize}
By hypothesis we have $Au = Mu + [d^\t u/(\vol G)]d \geq \lambda u$. 
The claim follows immediately by Theorem \ref{teo:main}.
\end{proof}

A close inspection of the preceding proof reveals that,
if $\lambda$ is not an eigenvalue of $A$ then 
we must have $\ell = \ell'+1$ and 
the previous inequalities become
$$
   \lambda_{\ell'}(M) >    \lambda_{\ell}(A) > 
   \lambda_{\ell}(M) = \lambda > \lambda_{\ell+1}(A)
   > \lambda_{\ell+1}(M) .
$$
Consequently the bound for the induced positive strong nodal domains in the theorem above becomes simply $\ell$.
The following corollary specializes the content of the preceding
theorem to eigenvectors associated to 
the algebraic modularity:

\begin{corollary}\label{cor:m(G)-domains}
Let $u$ be an eigenvector associated to $m(G)$
and oriented so that $d^\t u \geq 0$.
If $m(G)$ is simple and is not an eigenvalue of $A$
then $u$ induces exactly one positive (strong) nodal domain.
\end{corollary}

\begin{proof}
It suffices to observe that, if $m(G) = 0$ then
$u$ must be a multiple of $\uno$, and the claim is trivial.
On the other hand, if $m(G) \neq 0$ then $u^\t\uno = 0$,
so $u$ has at least two entries with different signs,
and the claim follows from 
the aforementioned interlacing inequalities and 
Theorem \ref{thm:modularity-nodal-theorem}.
\end{proof}

Unfortunately there exists no analogous of 
Theorem \ref{thm:modularity-nodal-theorem} 
for the negative nodal domains. This is illustrated by the following example.

\begin{example}
We produce a family of graphs, of arbitrarily large size,
to show that the number of
(unsigned) nodal domains induced by the leading eigenvector
of $M$ can be arbitrary, whilst there is exactly one
positive nodal domain, if signs are chosen 
as prescribed by Theorem \ref{thm:modularity-nodal-theorem}.
Consider a weighted star graph with loops 
on $n = m+1$ nodes, whose structure and adjacency matrix
are as follows:
$$
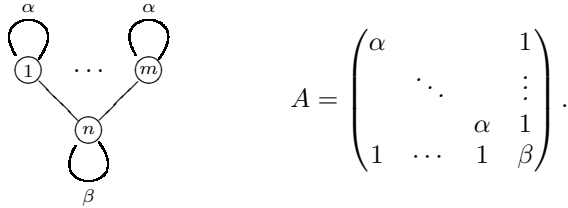
\begin{grafo}{4mm}
   (0,1)\NodoO[4mm]{1}{\scriptstyle 1}{R}{},
   (2,1)\Nodo[4mm]{2}{\ldots}{L}{},
   (4,1)\NodoO[4mm]{3}{\scriptstyle m}{R}{},
   (2,-1)\NodoO[4mm]{4}{\scriptstyle n}{R}{},
   \arcoD[]{4}{1}
   \arcoD[]{4}{3}
   \loopU[_{\alpha}]{1}
   \loopU[_{\alpha}]{3}
   \loopD[^{\beta}]{4}
\end{grafo}
\qquad\qquad
   A = \begin{pmatrix}
   \alpha & & & 1 \\ & \ddots & & \vdots \\
   & & \alpha & 1 \\ 1 & \cdots & 1 & \beta \end{pmatrix} .
$$
Hence, $\alpha$ and $\beta$ are the weights 
of the loops placed on 
the leaf nodes $1,\ldots,m$ and on the root node
$n = m+1$, respectively. 
In particular, the degree vector is 
$d = A\uno = (1+\alpha, \ldots, 1+\alpha, \beta + m)^\t$
and the volume is $\vol G = d^\t\uno = m(2+\alpha) + \beta$.
Straightforward computations leads to 
the conclusion that the eigenvalues of the modularity matrix $M$
are the following:
\begin{itemize}
\item $0$, with associated eigenvector $\uno$;

\item $\alpha$, with multiplicity $m-1$ and associated
eigenvectors $\uno_{\{1\}}-\uno_{\{j\}}$ for $j = 2,\ldots,m$;

\item $\bar \lambda = (\alpha\beta - m)(m+1)/\vol G$,
with associated eigenvector $\bar v = (-1,\ldots,-1,m)^\t$.
\end{itemize}
Observe that, when 
$$ 
   \alpha\beta - m > (\alpha+1)^2
$$ 
then $\bar \lambda$ is positive and dominant.
If in addition 
$\alpha + 1 \leq \beta + m$
then the vector $\bar v$ fulfills the inequality 
$d^\t\bar v \geq 0$, whence
the spectral clustering of the graph consists
of one positive nodal domain, given by node $n$,
and $m$ distinct, negative nodal domains,
given by the leaf nodes.
Both inequalities above are fulfilled when, for example, 
$\alpha = 1$ and $\beta = m+5$.

On the other hand, the Laplacian matrix of the same
graph is 
$$
   L = \begin{pmatrix}
   1 & & & -1 \\ & \ddots & & \vdots \\
   & & 1 & -1 \\ -1 & \cdots & -1 & m \end{pmatrix} ,
$$
independently on $\alpha$ and $\beta$.
Its smallest nonzero eigenvalue is $1$ and the associated eigenspace is set of all zero-sum vectors
that are orthogonal to 
the $n$-vector $(1,\ldots,1,0)$. Hence,
any spectral partitioning induced by a Fiedler vector
has exactly two weak nodal domains (which intersect
at the root node), whereas the number of (positive and negative)
strong nodal domains
can vary in the range $2,\ldots,m$.
\end{example}

Analogous examples can be built up using loopless,
unweighted graphs. Indeed, consider a
graph with $p+mq$ nodes
consisting of one clique with $p$ nodes
and $m$ copies of the clique with $q$ nodes.
Moreover, add $m$ edges connecting a fixed node of the
former subgraph with one node of each of the latter
subgraphs. The case with $p = 4$, $m = q = 3$
is shown hereafter:
$$
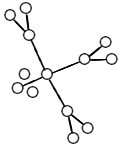
\begin{grafo}{10mm} 
( 0.41 ,  0.57)\NodoO[1.5mm]{1}{}{L}{}, 
( 0.02 ,  0.39)\NodoO[1.5mm]{2}{}{L}{}, 
( 0.22 ,  0.34)\NodoO[1.5mm]{3}{}{L}{}, 
( 0.11 ,  0.58)\NodoO[1.5mm]{4}{}{L}{}, 
( 0.91 ,  0.77)\NodoO[1.5mm]{5}{}{L}{}, 
( 1.19 ,  1.00)\NodoO[1.5mm]{6}{}{L}{}, 
( 1.27 ,  0.78)\NodoO[1.5mm]{7}{}{L}{}, 
( 0.68 ,  0.08)\NodoO[1.5mm]{8}{}{L}{}, 
( 0.95 , -0.14)\NodoO[1.5mm]{9}{}{L}{}, 
( 0.76 , -0.27)\NodoO[1.5mm]{10}{}{L}{}, 
( 0.17 ,  1.10)\NodoO[1.5mm]{11}{}{L}{}, 
( 0.13 ,  1.45)\NodoO[1.5mm]{12}{}{L}{}, 
(-0.08 ,  1.36)\NodoO[1.5mm]{13}{}{L}{}, 
\arcoD[]{2}{1} 
\arcoD[]{3}{1} 
\arcoD[]{4}{1} 
\arcoD[]{5}{1} 
\arcoD[]{8}{1} 
\arcoD[]{11}{1} 
\arcoD[]{1}{2} 
\arcoD[]{3}{2} 
\arcoD[]{4}{2} 
\arcoD[]{1}{3} 
\arcoD[]{2}{3} 
\arcoD[]{4}{3} 
\arcoD[]{1}{4} 
\arcoD[]{2}{4} 
\arcoD[]{3}{4} 
\arcoD[]{1}{5} 
\arcoD[]{6}{5} 
\arcoD[]{7}{5} 
\arcoD[]{5}{6} 
\arcoD[]{7}{6} 
\arcoD[]{5}{7} 
\arcoD[]{6}{7} 
\arcoD[]{1}{8} 
\arcoD[]{9}{8} 
\arcoD[]{10}{8} 
\arcoD[]{8}{9} 
\arcoD[]{10}{9} 
\arcoD[]{8}{10} 
\arcoD[]{9}{10} 
\arcoD[]{1}{11} 
\arcoD[]{12}{11} 
\arcoD[]{13}{11} 
\arcoD[]{11}{12} 
\arcoD[]{13}{12} 
\arcoD[]{11}{13} 
\arcoD[]{12}{13} 
\end{grafo} 
$$
Under appropriate conditions on 
the parameters $p$, $q$, and $m$ the leading eigenvector of $M$ splits the graph
into the $m+1$ cliques, each belonging to
a different nodal domain; the (unique) positive nodal domain
being the clique having order $p$.
Computer experiments show that those conditions are met e.g.,
for $p = 4$, $q = 3$, and $m = 2,\ldots,11$.

\section{Upper bounds on the graph modularity}

In the preceding sections we have understood modularity
as a functional defined over arbitrary subsets of $V$. 
For the purposes of community detection problems, it is convenient to extend the previous definition to arbitrary partitions. In fact,
Newman and Girvan original definition of the modularity of 
a partition $\matheul{P} = \{S_1,\ldots,S_k\}$ of $V$, 
see Equation (5) in \cite{newman-girvan}, 
can be expressed in our notations as
\begin{equation}  \label{eq:defq(P)}
   q(\matheul{P}) = 
   \frac{1}{\vol G} \sum_{i=1}^k Q(S_i) =
   \frac{1}{\vol G} \sum_{i=1}^k \uno_{S_i}^\t M\uno_{S_i} .
\end{equation}
The normalization factor $1/\vol G$ is purely conventional and has been included by the authors for compatibility with previous works, 
to settle the value of $q(\matheul{P})$ in a range independent of $G$.
That definition has been introduced as a merit function to quantify the strength of the
community structure defined by $\matheul{P}$.
In the earliest 
community detection algorithm, the function $q(\matheul{P})$
is optimized by a hierarchical clustering method.
Subsequent improvements of that algorithm
maintain essentially the original approach, see \cite{Louvain_method}.
The use (and the definition itself) of the modularity matrix 
to compute the modularity of a paritioning 
has been introduced successively
in \cite{newman-eigenvectors,newman-modularity}.

As recalled in the Introduction, in the community detection problem
one has no preliminary indications on the number and size of possible communities inside $G$. Hence, 
it is natural to introduce 
the number 
$$
   q_G = 
   \max_{\matheul{P}} q(\matheul{P}) ,
$$
where the maximum is taken over all nontrivial partitions of $V$,
and try to bound it in terms of spectral properties of $M$ only.

\begin{remark}   \label{rem:5.1}
An optimal partition $\matheul P_\ast = \{S_1, \dots, S_k\}$, 
that is, 
a partition such that $q_G = q(\matheul{P}_\ast)$,
has the property that if any two subsets are merged then the overall modularity does not increase. This does not imply that $Q(S_i) > 0$ for all $i=1,\ldots,k$,
even if $G$ is divisible. Nevertheless,
if $Q(S) > 0$ for some $S\subset V$ then 
$q_G \geq q(\{S,\bar S\}) > 0$, so that the condition 
$q_G > 0$ is equivalent to say that $G$ is divisible.
\end{remark}

\medskip
In the $k=2$ case we have
$\matheul{P} = \{S,\bar S\}$ for some $S\subset V$.
Since $Q(S) = Q(\bar S)$, for notational simplicity, we can write $q(S)$
in place of $q(\matheul{P})$. Correspondingly, we also consider
the quantity
$$
   q_G' = 
   \max_{S\subset V} q(S) ,
$$
whose computation corresponds to the identification of a 
set $S$ or, 
equivalently, a cut $\{S,\bar S\}$
with maximal modularity.
We prove hereafter a very general upper bound for $q_G'$
in terms of $m(G)$; a lower bound is considered in the 
forthcoming Corollary \ref{cor:cheeger}, 
under additional hypotheses.

\begin{theorem}   \label{thm:qprimeG}
Let $\langle d \rangle = \vol G/n$ be the average degree 
in $G$. Then,
$$
   q_G' \leq \frac{m(G)}{2\,\langle d \rangle} .
$$
\end{theorem}

\begin{proof}
Since $Q(S) = Q(\bar S)$, for any $S\subseteq V$ we have
by definition
\begin{equation}   \label{eq:q(S)}
q(S) = \frac{1}{\vol G} \big( Q(S) + Q(\bar S) \big)
=\frac{2\, Q(S)}{\vol G} .
\end{equation}
Let $S$ be a set maximizing $q(S)$. 
From Lemma \ref{lem:g_G'} we obtain
$$
   q_G' = \frac{2\, Q(S)}{\vol G}
   \leq \frac{2\, m(G)}{\vol G} \frac{|S||\bar S|}{n} 
   \leq \frac{m(G)}{\vol G} \frac{n}{2} ,
$$
since $|S||\bar S|$ is upper bounded by $n^2/4$
for any $S$. 
\end{proof}

\medskip

In what follows, we prove a result analogous to the preceding theorem
but for the number $q_G$.
For clarity of exposition, we derive firstly a preliminary result:

\begin{lemma}   \label{lem:bhatia}
Let $A$ and $B$ be two symmetric matrices of order $n$,
with eigenvalues $\lambda_1(A) \ge \ldots \ge\lambda_n(A)$
and $\lambda_1(B) \ge \ldots \ge\lambda_n(B)$, respectively.
Then,
$$
\mathrm{trace}(AB) \le \sum_{i=1}^n \lambda_i(A)\lambda_i(B) .
$$
\end{lemma}

\begin{proof}
The claim can be derived easily from the 
Hoffman-Wielandt inequality \cite[Thm.\ VI.4.1]{bhatia}
$$
   \sum_{i=1}^n (\lambda_i(A) - \lambda_i(B))^2 
   \leq \|A-B\|_F^2 ,
$$
using the expansion $\|A-B\|_F^2 = \|A\|_F^2 +
\|B\|_F^2 - 2\,\mathrm{trace}(AB)$ and the equality
$\|A\|_F^2 = \sum_{i=1}^n \lambda_i(A)^2$.
\end{proof}

Consider an arbitrary partition $\matheul P=\{S_1, \dots, S_k\}$ of the node set $V$.  
Assume for simplicity that  
$|S_i|\geq |S_{i+1}|$ for $i=1, \dots, k-1$. 
Introduce the $n\times k$ ``index matrix'' 
$Z = [\uno_{S_1} \cdots \uno_{S_k}]$ and
define $B=ZZ^\t =\sum_{i=1}^k \uno_{S_i}\uno_{S_i}^\t$. Then $B$ has rank $k$ and the cardinalities $|S_i|$ are its nonzero eigenvalues in nonincreasing order.
Recalling that, for arbitrary matrices $A$ and $B$ 
it holds $\mathrm{trace}(AB) = \mathrm{trace}(BA)$, 
\eqref{eq:defq(P)} can be rewritten as follows:
$$ 
   q(\matheul P) = 
   \frac{1}{\vol G}\text{trace}(Z^\t MZ) = 
   \frac{1}{\vol G}\text{trace}(B M) .
$$ 
With the help of Lemma \ref{lem:bhatia}
we immediately get
$$ 
   q(\matheul P) 
   \leq \frac{1}{\vol G}\sum_{i=1}^k
   |S_i|\lambda_i(M)\leq \frac{n}{\vol G} \lambda_1(M) .
$$ 
The latter bound, which does not depend on $\matheul{P}$, 
can be improved as follows:

\begin{theorem}
For any graph $G$,
$$
   q_G \leq \frac{n-1}{\vol G }m(G) .
$$
\end{theorem}

\begin{proof}
Let $\matheul{V} = \langle\uno\rangle^\perp$ 
and let $V$ be any 
matrix whose columns 
form an orthonormal basis of $\matheul{V}$.
Observe that $W = VV^\t$ is the orthogonal projector onto 
$\matheul{V}$, that is, $W = I - \frac{1}{n}\uno\uno^\t$. 
Moreover, $\lambda_1(V^\t MV) = m(G)$ by \eqref{eq:def_m(G)}.

Let $\matheul P_\ast=\{S_1, \dots, S_k\}$ 
be an optimal partition of $G$, that is 
$q_G = q(\matheul P_\ast)$, and let again
$Z = [\uno_{S_1} \cdots \uno_{S_k}]$.
Since $M = WMW$, 
Lemma \ref{lem:bhatia} leads us to
\begin{align*}
   \text{trace}(Z^\t MZ) & = \text{trace}(Z^\t WMWZ)
   = \text{trace}((V^\t ZZ^\t V)(V^\t MV)) \\  
   & \leq \sum_{i=1}^n
   \lambda_i(V^\t ZZ^\t V)\lambda_i(V^\t MV) \\
   & \leq \mathrm{trace}(V^\t ZZ^\t V)\, m(G) 
   = \mathrm{trace}(Z^\t W Z)\, m(G) .
\end{align*}
From $W = I - \frac{1}{n}\uno\uno^\t$, 
letting $z = Z^\t \uno = (|S_1|,\ldots,|S_k|)^\t$
we obtain
$$
   Z^\t W Z = Z^\t Z - \frac{1}{n}Z^\t \uno\uno^\t Z
   = \mathrm{Diag}(|S_1|,\ldots,|S_k|) - \frac{1}{n}zz^\t .
$$
Owing to the fact that $\|z\|_2 \geq \|z\|_1/\sqrt{n} = \sqrt{n}$,
we have $\mathrm{trace}(Z^\t W Z) = n - \|z\|_2^2/n \leq n-1$.
It suffices to collect terms,
and the proof is complete.
\end{proof}

%
%
%

\section{How many modules?}   \label{sec:6} 

Based on rather informal arguments, Newman
claims in \cite[Sect.\ B]{newman-eigenvectors}
that the number of positive eigenvalues of $M$
is related to the number of communities recognizable
in the graph $G$,
tightening the connection
between spectral properties of $M$ and the community structure 
of the network it describes. More precisely,
the author argues that the number of positive eigenvalues, 
plus $1$,
is an upper bound on the number of communities
that can be recognized in $G$.
In this section we prove 
various results supporting that conclusion,
that culminate in the subsequent Corollary \ref{cor:best_P}.

Already in Remark \ref{rem:m(G)} we noticed that the existence
of a subgraph $S$ having positive modularity implies
that $M$ has at least one positive eigenvalue.
By the way, if $Q(S)>0$ then also $Q(\bar S)>0$,
therefore two modules (according to Definition \ref{def:community}) give rise to one positive eigenvalue.
The forthcoming theorem proves that, if $G$
has $k$ subgraphs that are
well separated and sufficiently rich in internal edges,
then $M$ has at least $k-1$ positive eigenvalues.

\begin{theorem}   \label{teo:communities1}
Let $S_1,\ldots,S_k$ be pairwise disjoint subsets of $V$,
with $k \ge 1$,
such that $\vol(S_i) \le \frac12\vol G$ and 
$e_{\mathrm{in}}(S_i) > e_{\mathrm{out}}(S_i)$.
Then, $Q(S_i) > 0$ and
$M$ has at least $k-1$ positive eigenvalues.
\end{theorem}

\begin{proof}
Firstly we observe that, owing to the stated hypotheses,
the sets $S_1,\ldots, S_k$ have positive modularity.
Indeed, for $i = 1,\ldots, k$ we have
$\vol G \le 2 \, \vol \overline{S_i}$, whence
\begin{align*}
   Q(S_i) \frac{\vol G}{\vol \overline{S_i}} & = 
   \vol S_i - e_{\mathrm{out}}(S_i) \frac{\vol G}{\vol \overline{S_i}} \\
   & = e_{\mathrm{in}}(S_i) - e_{\mathrm{out}}(S_i) 
   \bigg(  \frac{\vol G}{\vol \overline{S_i}} - 1 \bigg) \\
   & \ge e_{\mathrm{in}}(S_i) - e_{\mathrm{out}}(S_i) > 0 .
\end{align*}
Let $Z = [\uno_{S_1} \cdots \uno_{S_k}]$ and 
consider the $k\times k$ matrix $C = Z^\t AZ$. 
For $i,j=1,\ldots,k$ we have
$$
   C_{ij} = \uno_{S_i}^\t A\uno_{S_j} 
   = \begin{cases}
   e_{\mathrm{in}}(S_i) & \quad i=j \\
   e(S_i,S_j) & \quad i\ne j . \end{cases}
$$ 
The matrix $C$ is symmetric, nonnegative, and 
(strongly) diagonally dominant. Indeed, 
owing to the fact that the $S_j$'s are pairwise disjoint, 
and $E(S_i,\overline{S_i}) \supseteq \cup_{j\ne i} E(S_i,S_j)$,
for $i = 1,\ldots,k$ we have
$$
   C_{ii} = e_{\mathrm{in}}(S_i) > e_{\mathrm{out}}(S_i)
   \ge \sum_{j\ne i} e(S_i,S_j) =
   \sum_{j\ne i} C_{ij} .
$$
As a result, by Gershgorin theorem, $C$ is positive definite. 

Introduce the diagonal matrix 
$\Delta = \mathrm{Diag}(\sqrt{|S_1|},\ldots,\sqrt{|S_k|})^{-1}$. Owing to the orthogonality of the columns of $Z$, 
the matrix $\hat Z = Z\Delta$ has orthonormal columns.
By Sylvester's law of inertia, also the matrix $\Delta C\Delta = \hat Z^\t A \hat Z$ is positive definite. 
From eigenvalue interlacing inequalities \eqref{eq:interlacing},
$$
   \lambda_k(A) \geq \lambda_k(\hat Z^\t A \hat Z)
   = \lambda_k(\Delta C\Delta) > 0 .
$$ 
Finally, using \eqref{eq:Weyl} we conclude
$\lambda_{k-1}(M) \ge \lambda_k(A) > 0$
and the proof is complete.
\end{proof}

In the subsequent theorem we apply an argument
similar to the one in the abovementioned result
directly to the matrix $M$ instead of $A$,
as intermediate step. Before that,
it is convenient to introduce an auxiliary notation.

Let $S_1$ and $S_2$ two disjoint subsets of $V$.
We define their {\em joint modularity} as
$$
   Q(S_1,S_2) = e(S_1,S_2) - 
   \frac{\vol S_1 \vol S_2}{\vol G} .
$$   
Its absolute value $|Q(S_1,S_2)|$
is sometimes referred to as {\em discrepancy between $S_1$ and $S_2$},
see e.g., \cite[\S 5.2]{chung} and \cite{expanders}.
The following properties are straightforward:
\begin{enumerate}
\item
Clearly, $Q(S_1,S_2) = Q(S_2,S_1)$
and $Q(S) = Q(S,\bar S)$. Furthermore, 
we can express the joint modularity of $S_1$ and $S_2$
equivalently as 
$$
   Q(S_1,S_2) = \uno_{S_1}^\t M \uno_{S_2} .
$$
Note that $Q(S_1,S_2)$
is the difference between the overall weight of edges bridging
$S_1$ and $S_2$ and 
its value 
as expected by the (weighted) Chung-Lu model.
\item
From the equation 
$(\uno_{S_1}+\uno_{S_2})^\t M(\uno_{S_1}+\uno_{S_2}) = 
\uno_{S_1}^\t M \uno_{S_1}+\uno_{S_2}^\t M\uno_{S_2}
+ 2 \uno_{S_1}^\t M\uno_{S_2}$ 
we have 
$$
   Q(S_1 \cup S_2) = Q(S_1) + Q(S_2) + 2 Q(S_1,S_2) .
$$
In particular, $Q(S_1,S_2) > 0$
if and only if $Q(S_1 \cup S_2) > Q(S_1) + Q(S_2)$.
Hence, when looking for an optimal partitioning of $G$
into modules, it is necessary that the 
joint modularity of any two subsets is $\le 0$,
otherwise, we can increase the overall modularity 
by merging two subgraphs into one.
\end{enumerate}

The forthcoming theorem proves that,
under ample hypotheses,
the number of positive eigenvalues of $M$, plus $1$,
is actually an upper bound for the cardinality of any
partition of $G$ into modules 
such that if any two subsets are merged
then the overall modularity 
does not increase.

\begin{theorem}   \label{teo:communities2}
Let $\matheul{P} = \{S_1,\ldots,S_k\}$ be a partition of $V$,
with $k \ge 2$,
such that $Q(S_i) > 0$ and 
$Q(S_i,S_j) \leq 0$ for $i\ne j$.
Consider the matrix $C$ such that
$C_{ii} = Q(S_i)$ and 
$C_{ij} = Q(S_i,S_j)$ for $i\ne j$.
If $C$ is irreducible then $M$ has at least $k-1$ positive eigenvalues.
\end{theorem}

\begin{proof}
Consider the matrices $Z = [\uno_{S_1} \cdots \uno_{S_k}]$, 
$\Delta = \mathrm{Diag}(|S_1|,\ldots,|S_k|)^{-1/2}$
and $\hat Z = Z\Delta$.
Then, $C = Z^\t MZ$. 
Furthermore, $C$ is weakly diagonally dominant. Indeed,
$$
   \sum_{j=1}^k C_{ij} =
   \uno_{S_i}^\t M \sum_{j=1}^k \uno_{S_j} =   
   \uno_{S_i}^\t M \uno = 0 .
$$
Using Gershgorin theorem we deduce $C$ is a symmetric 
positive semidefinite matrix, with a zero eigenvalue which 
is associated to the eigenvector $\uno$. 
For a sufficient large $\alpha > 0$ the matrix
$B = \alpha I - C$ is entrywise nonnegative and irreducible.
Hence, by Perron-Frobenius theory, its largest eigenvalue
is simple. Since the eigenspaces of $B$ and $C$
coincide, the zero eigenvalue of $C$ must be simple.

We deduce that 
$C$ has $k-1$ positive eigenvalues.
The same conclusion holds true also for the matrix 
$\Delta C\Delta = \hat Z^\t M \hat Z$, by Sylvester's law of inertia. 
Finally, eigenvalue interlacing inequalities 
\eqref{eq:interlacing} imply
$\lambda_{k-1}(M) \ge \lambda_{k-1}(\hat Z^\t M \hat Z)
= \lambda_{k-1}(\Delta C\Delta) > 0$,
and the proof is complete.
\end{proof}


Note that, in the preceding theorem, irreducibility of $C$
is verified in particular when $Q(S_i,S_j) < 0$ for all $i\neq j$.
That condition is fulfilled 
by any partition maximizing $q_G$ which contains 
the least number of sets among all such partitions
(otherwise we can reduce their
number by merging pairs whose joint
modularity is zero without decreasing the overall modularity).
We get an immediate corollary:


\begin{corollary}   \label{cor:best_P}
Let $\matheul P_\ast$ be a minimal cardinality
partition with $q_G = q(\matheul P_\ast)$, 
interely made by modules. 
Then $\matheul P_\ast$ contains no more than $k+1$ sets, being $k$ the number of positive eigenvalues of $M$.
\end{corollary}

The following example, which is inspired by a popular
benchmark in the community detection literature,
shows that this result is optimal:

\begin{example}[Circulant ring of clusters]
Given integers $p>2$ and $q>2$,
consider the graph consisting of $n = pq$ vertices,
partitioned as $\matheul{P} = \{S_1,\ldots,S_p\}$;
every $G(S_i)$ is a clique of order $q$; the cliques are arranged 
circularly, and every node of $S_i$ is connected to the
corresponding node of the two neighboring cliques
by an edge whose weight is $\gamma \in(0,1)$
(so that the generalized degree of each node is 
$q-1+2\gamma$).
In this graph, the $p$ cliques have postive modularity,
and in fact are clearly recognizable as 
``communities''. We show hereafter that if 
$0 < \gamma \leq 1/2$ the modularity matrix
of this graph has exactly $p-1$ positive eigenvalues.  

With a natural numbering of the nodes,
the adjacency matrix can be expressed as 
a block circulant matrix with circulant blocks,
$$
   A = \begin{pmatrix}
   B_q & \gamma I & & \gamma I \\
   \gamma I & B_q & \ddots & \\
   & \ddots & \ddots & \gamma I \\
   \gamma I & & \gamma I & B_q \end{pmatrix} ,
$$
where $B_q = \uno\uno^\t - I$ is the adjacency matrix
of a $q$-order clique; and the corresponding
modularity matrix is $M = A - c\uno\uno^\t$, with 
$c = (q-1+2\gamma)/n$, which is still block circulant
with circulant blocks and, furthermore, 
simultaneously diagonalizable with $A$. 
In fact, let $C_p \equiv(c_{ij})$ be the $p\times p$ symmetric circulant matrix such that 
$c_{ij} = 1$ if $|i-j| = 1\mod p$ and $c_{ij} = 0$ otherwise.
Denoting by $F_k$ the unitary Fourier matrix of order $k$,
we have the spectral factorizations
\begin{align*}
   C_p & = F_p 
   \mathrm{Diag}(\lambda^{(p)}_1,\ldots,\lambda^{(p)}_p) F_p^* ,
   \quad \hbox{with }\lambda^{(p)}_j = 2\cos(2\pi(j-1)/p) , \\
   B_q & = F_q \mathrm{Diag}(q-1, -1, \ldots , -1)
    F_q^* .
\end{align*}
Making use of the Kronecker (tensor) product,
the adjacency matrix $A$ admits the decomposition 
$A = I\otimes B_q + \gamma C_p \otimes I$, whence it is
diagonalized by $F_p\otimes F_q$.
Consequently,
the eigenvalues of $A$ are readily computed as follows:
\begin{itemize}
\item[a)]
$q-1 + \gamma \lambda^{(p)}_j$ for $j = 1,\ldots,p$,
each of them having multiplicity $1$; and
\item[b)]
$\gamma \lambda^{(p)}_j - 1$ for $j = 1,\ldots,p$,
each of them having multiplicity $q-1$.
\end{itemize}
A careful observation reveals that, 
if $\gamma \leq \frac12$ then the $p$ largest eigenvalues 
of $A$ are precisely the numbers in the preceding item a),
which are positive; and
the remaining eigenvalues are $\leq 0$.

The eigenvalues of $M$ coincide with 
those of $A$ with the exception of the largest one, 
which is annihilated by the rank-one correction $A - M = c\uno\uno^\t$.
Consequently, the matrix $M$ has at least $p-1$
positive eigenvalues; for all $0<\gamma\leq\frac12$, 
they are exactly $p-1$, that is, the number of 
``communities'' minus one. 
It is interesting to note that eigenvectors associated to these eigenvalues lie in the span of $\Re(f_k)\otimes \uno$
and $\Im(f_k)\otimes \uno$, where $f_k$ is the $k$-th
column of $F_p$; in particular, they are constant 
within each clique. Furhermore, for any two distinct 
integers $i,j = 1,\ldots,p$,
one such eigenvector assumes opposite signs 
on $S_i$ and $S_j$,
so that communities in this graph are demarcated 
precisely by modularity nodal domains
associated to positive eigenvalues of $M$.
\end{example}

\section{A Cheeger-type inequality}

Let $G = (V,E)$ an unweighted graph.
The number
$$   
   h_G = \min_{\substack{S\subset V\\0<|S| \leq \frac{n}{2}}}
   \frac{|E(S,\bar S)|}{|S|} 
$$
is one of best known topological invariants of $G$, 
as it establishes a wealth of deep and important relationships with various areas of mathematics \cite{chung,expanders}.
Its connection with graph partitioning,
and discrete versions of the isoperimetric problem, is apparent. Hence, it is of no surprise that 
various relationships have been uncovered between $h_G$
and $a(G)$, also under slightly different definitions.

The bound $h_G \geq a(G)/2$
can be obtained by rather elementary arguments. 
Various converse inequalities exist and  
bear the name of {\em Cheeger inequality}, 
analogously to a classical result in Riemannian geometry
that relates the solution of the isoperimetric problem
to the smallest positive eigenvalue of the Laplacian differential operator
on manifolds.
For example, it is known that if $G$ is a $k$-regular graph 
(that is, $d_i = k$ for $i=1,\ldots,n$) then
$h_G \leq \sqrt{2ka(G)}$ \cite[Thm.\ 4.11]{expanders}.
In the forthcoming Corollary \ref{cor:cheeger} we provide a Cheeger-type inequality 
between modularity and algebraic modularity of a regular graph.
Although practical graphs are seldom regular, 
that hypothesis is important to obtain a converse result
to Theorem \ref{thm:qprimeG}.

\begin{theorem}
Let $G=(V,E)$ be a connected, $k$-regular graph,
and let $f$ be an eigenvector associated to $m(G)$:
$Mf = m(G)f$.
Let $w_1 \ge w_2 \ge \ldots \ge w_n$ be the values of
$f_1,\ldots,f_n$ sorted in nonincreasing order.
Introduce the sets
$$
   S_i = \{j: f_j \le w_i\} , \qquad i = 1,\ldots,n,
$$
and let $Q_\star = \max_i Q(S_i)$.
Then,
$$
   Q_\star \ge \frac{1}{w_1-w_n}\bigg(
   \frac{k}{2} \|f\|_1 
   - \|f\|_2 \sqrt{(k-m(G))\frac{kn}{2}} \bigg) .
$$
\end{theorem}

\begin{proof}
We start by noticing that $f$ is a Fiedler vector of $G$.
Indeed, if $G$ is $k$-regular then 
the matrix $L_0$ in \eqref{eq:M_L} becomes 
$L_0 = kI - (k/n)\uno\uno^T$ whence $L_0f = kf$; 
moreover, from the equation $m(G) = k-a(G)$ and the decomposition  
\eqref{eq:M_L} we obtain 
$Lf = a(G)f$, that is, $f$ is a Fiedler vector.

Consider the quantity
$$
   \sigma = \sum_{i\sim j}|f_i - f_j| ,
$$
where the sum runs on the edges of $G$, every edge being counted only once.
By Cauchy-Schwartz inequality and \eqref{eq:vLv}, 
\begin{align*}
   \sigma & \le 
   \sqrt{\sum_{i\sim j} (f_i-f_j)^2}
   \sqrt{\sum_{i\sim j} 1} =  \|f\|_2 \sqrt{a(G)\frac{\vol G}{2}} 
   = \|f\|_2 \sqrt{(k-m(G))\frac{kn}{2}} . 
\end{align*}   
For ease of notation, we re-number the vertices of $G$ 
so that $f_1\ge f_2 \ge \ldots \ge f_n$.
In this way, the sets $S_1,\ldots,S_n$ introduced in the claim 
are given by $S_i = \{1,\ldots,i\}$. 
Furthermore, the edge boundary $\partial S_i$ is the set of
all edges having one vertex in $\{1,\ldots,i\}$ and the other
in $\{i+1,\ldots,n\}$.
Let $Q_\star = \max_i Q(S_i)$. Using
$$
   |\partial S_i| \ge 
   \frac{\vol S_i\, \vol \overline S_i}{\vol G} 
   - Q_\star = 
   \frac{k\, i (n-i)}{n} - Q_\star 
$$
we obtain  
\begin{align*}
   \sigma & = \sum_{\substack{i\sim j \\ i < j}} (f_i-f_j)
   = \sum_{\substack{i\sim j \\ i < j}} 
   \sum_{\ell = i}^{j-1} (f_\ell-f_{\ell+1}) 
   = \sum_{i = 1}^{n-1} 
   (f_i-f_{i+1}) \cdot |\partial S_i| \\  
   & \ge k \sum_{i = 1}^{n-1} 
   (f_i-f_{i+1})\frac{i (n-i)}{n}  -
   Q_\star \sum_{i = 1}^{n-1} 
   (f_i-f_{i+1})  \\
   & = k \sum_{i = 1}^{n} 
   f_i\frac{n+1-2i}{n} -
   Q_\star (f_1-f_{n})  \\
   & = \frac{2k}{n} \sum_{i=1}^n \sum_{j=1}^i f_j 
   - Q_\star (f_1-f_{n}) .
\end{align*}
The last passages are obtained by collapsing the
telescopic sums, rearranging terms, 
and exploiting the equality $\sum_{i=1}^n f_i = 0$.
Now, let $m$ be an integer such that
\begin{equation}   \label{eq:f1fn}
   f_1 \ge \ldots \ge f_m \ge 0 \ge f_{m+1} \ge \ldots \ge f_n .
\end{equation}
Owing to the fact that
$\sum_if_i = 0$ we have 
$$
   \max_i \sum_{j=1}^i f_j = \sum_{j=1}^m f_j = 
   \frac12(f_1+\cdots +f_m + |f_{m+1}| + \cdots + |f_n|) = 
   \frac12 \|f\|_1 .
$$
Introduce the notation $F_i = \sum_{j=1}^i f_j$.
By virtue of the inequalities \eqref{eq:f1fn},
for all $j = 0,\ldots,m$ and $k = 0,\ldots,n-m$ we have
$$
   F_j + F_{m-j} \geq F_m , \qquad
   F_{m+k} + F_{n-k} \geq F_m .
$$
Thus we obtain
\begin{align*}
   \sum_{i=1}^n \sum_{j=1}^i f_j =
   \sum_{i=1}^n F_i & = 
   \frac12 \sum_{j=1}^{m-1} (F_j+F_{m-j}) + \frac12 \sum_{k=0}^{n-m} (F_{m+k}+F_{n-k}) \\
   & \geq 
   \frac{n}{2} F_m = \frac{n}{4}\|f\|_1 .
\end{align*}
Putting it all together we get
$$
   \frac{k}{2} \|f\|_1 
   - Q_\star (f_1-f_{n})
   \le \sigma \le \|f\|_2 \sqrt{(k-m(G))\frac{kn}{2}} ,
$$
whence we obtain the claim. 
\end{proof}

\begin{corollary}   \label{cor:cheeger}
If $G=(V,E)$ is a connected, $k$-regular graph then
$$
   \frac{1}{2n} - \sqrt{\frac{k-m(G)}{2k}} 
   \leq q_G' \leq   
   \frac{m(G)}{2k}  .
$$
\end{corollary}

\begin{proof}
The upper bound directly follows by Theorem \ref{teo:communities1} in the $k$-regular case. In the notations of the preceding theorem we observe that
$w_1-w_{n} = w_1+|w_n| \le 2 \|f\|_\infty$.
Using the inequality 
$\|f\|_2 \le \sqrt{n}\|f\|_\infty$ 
and $\|f\|_\infty \le \|f\|_1$ we obtain
$$
   Q_\star \ge \frac{k\|f\|_1}{4\|f\|_\infty} 
   - \frac{\|f\|_2}{2\|f\|_\infty}\sqrt{(k-m(G))\frac{kn}{2}}
   \ge \frac{k}{4} 
   - \frac{\sqrt{n}}{2}\sqrt{(k-m(G))\frac{kn}{2}} .
$$
To complete the proof it is sufficient to 
observe that, in view of \eqref{eq:q(S)},
we have $q_G' \geq 2Q_\star/\vol G = 2Q_\star/(kn)$.
\end{proof}

\section{Concluding remarks}

In this paper we have studied the community detection problem trough modularity optimisation from an uncommon algebraic point of view. In particular we have tried to propose 
popular concepts from complex networks and physics literatures in a mathematical formalism involving mainly linear algebra and matrix theory. 

We introduce the concept of algebraic modularity 
of a graph, allowing to clarify the difference between indivisible graph and algebraically indivisible graph, often  used with not much attention interchangeably one with the other.
We focus our attention on the nodal domains induced by the eigenvectors of the modularity matrix and we derive a nodal domain theorem for such eigenvectors, in complete analogy with the well known Fiedler vector theorem for the Laplacian matrix \cite{fiedler-vector}, and some further developments proposed 
more recently in \cite{nodal-domain-theorem, duval-nodal-domains, powers-graph-eigenvector}.
However, unlike in the Laplacian case,
nodal domains arising with modularity matrices are naturally endowed by a sign, with different properties for positive and negative nodal domains. 

Then we consider the possible relationship between 
the number of modules 
in $G$ and the number of positive eigenvalues of its modularity matrix. Newman
claimed in \cite{newman-eigenvectors}
that the number of positive eigenvalues of $M$
is related to the number of communities recognizable
in the graph $G$, but his claim was based on rather informal arguments. Our analysis of $M$ instead tries to support this claim showing, in particular, that the presence of communities in $G$ implies that the spectrum of $M$ at least partially lies on the positive axis. We would point out here that a reverse implication is realistic and desirable, but is still an open problem.

Finally we focus the attention 
on Cheeger-type inequalities,
discovering 
that a nice estimate elapses between modularity and algebraic modularity of $G$.
At present, our result is limited to regular graphs;
its possible extension to more general graphs 
seems to be a major task and is left as an open problem.

As the importance of the community detection problem is apparent, and mo\-du\-la\-ri\-ty-based techniques are by far the most popular in this ambit, we believe that the modularity matrix $M$ could be considered as a relevant matrix in algebraic graph theory, together with adjacency and Laplacian matrices.
The results we obtain give rise to a first spectral graph analysis aimed at the problems of existence, estimation and localization of optimal subdivisions of the graph into communities.
Our results adhere to modularity-related definitions 
borrowed from current literature. Probably,
modified (maybe, ``normalized'') versions of modularity
matrices and functions
may lead to conclusions different from those presented here.

\section*{Acknowledgements}
The authors appreciate two anonymous referees for their useful comments and suggestions, in particular, those of 
including Remark 1.1 and
Example 6.4 in the final version of this paper.

\bibliography{./bib_preprint}
\bibliographystyle{plain}

\end{document}